\documentclass{amsart}

\usepackage{xcolor}

\usepackage{latexsym,amsmath,amssymb,epic}
\usepackage{latexsym,amssymb,epic}
\usepackage{amsmath,amsthm}
\usepackage{color}

\usepackage{soul}

\usepackage{hyperref}
\hypersetup{
	colorlinks=true, 
	linktoc=all, 
	linkcolor=blue} 

\newcommand{\seqnum}[1]{\href{https://oeis.org/#1}{\rm \underline{#1}}}

\newtheorem{theorem}{Theorem}[section]
\newtheorem{lemma}[theorem]{Lemma}
\newtheorem{corollary}[theorem]{Corollary}

\allowdisplaybreaks

\begin{document}
	
\title[Partitions with Odd Parts Repeated at Most Twice]{Elementary Proofs of Two Congruences for Partitions with Odd Parts Repeated at Most Twice}

\author{James A. Sellers}
\address{Department of Mathematics and Statistics, University of Minnesota Duluth, Duluth, MN 55812, USA}
\email{jsellers@d.umn.edu}

	
	
\maketitle
\begin{abstract}

In a recent article on overpartitions, Merca considered the auxiliary function $a(n)$ which counts the number of partitions of $n$ where odd parts are repeated at most twice (and there are no restrictions on the even parts).  In the course of his work, Merca proved the following:  
For all $n\geq 0$, 
\begin{align*}
a(4n+2) &\equiv 0 \pmod{2}, \textrm{\ \ and} \\
a(4n+3) &\equiv 0 \pmod{2}.
\end{align*}
Merca then indicates that a classical proof of these congruences would be very interesting. The goal of this short note is to fulfill Merca's request by providing two truly elementary (classical) proofs of these congruences.
\end{abstract}

\section{Introduction}  

A {\it partition} of a positive integer $n$ is a finite non--increasing sequence of positive integers $\lambda_1 \geq \lambda_2 \geq \dots \geq \lambda_k$ such that $\lambda_1 + \lambda_2 + \dots + \lambda_k = n.$  We refer to the integers $\lambda_1, \lambda_2, \dots, \lambda_k$ as the {\it parts} of the partition.  For example, the number of partitions of the integer $n=4$ is 5, and the partitions counted in that instance are as follows:  
$$4, \ \ \ 3+1, \ \ \ 2+2, \ \ \ 2+1+1, \ \ \ 1+1+1+1$$
The number of partitions of $n$ is denoted $p(n)$, and the corresponding generating function is given by 
\begin{equation}
\label{genfnp}
\sum_{n\geq 0}p(n)q^n = \frac{1}{f_1} 
\end{equation}
where $f_r = (1-q^r)(1-q^{2r})(1-q^{3r})\dots$ for any positive integer $r.$   

An {\it overpartition} of a positive integer $n$ is a partition of $n$ wherein the first occurrence of a part may be overlined.  
For example, the number of overpartitions of $n=4$ is 14 given the following list of overpartitions of 4:  
$$
4, \ \ \ \overline{4}, \ \ \ 3+1, \ \ \ \overline{3}+1, \ \ \ 3+\overline{1}, \ \ \ \overline{3}+\overline{1}, 
$$
$$
2+2, \ \ \  \overline{2} + 2, \ \ \ 2+1+1, \ \ \  \overline{2}+1+1, \ \ \  2+\overline{1}+1, \ \ \  \overline{2}+\overline{1}+1, 
$$
$$
1+1+1+1, \ \ \ \overline{1}+1+1+1
$$
The number of overpartitions of $n$ is often denoted $\overline{p}(n),$ so from the example above we see that $\overline{p}(4) = 14.$  
As noted by Corteel and Lovejoy \cite{CL}, the generating function for $\overline{p}(n)$ is given by 
\begin{equation*}
\sum_{n\geq 0} \overline{p}(n)q^n = \frac{f_2}{f_1^{2}}.
\end{equation*}

In recent work on overpartitions \cite{MercaAM}, Merca considered the auxiliary function $a(n)$ which counts the number of partitions of weight $n$ wherein no part is congruent to 3 modulo 6.  From this definition, it is clear that the generating function for $a(n)$ is given by 
\begin{equation}
\label{genfn_a}
\sum_{n\geq 0} a(n)q^n 
=
\prod_{i\geq 1} \frac{(1-q^{6i-3})}{(1-q^i)} 
= 
\prod_{i\geq 1} \frac{(1-q^{6i-3})}{(1-q^i)} \cdot \frac{(1-q^{6i})}{(1-q^{6i})} 
= 
\frac{f_3}{f_1f_6}.
\end{equation}
The function $a(n)$ also counts the number of partitions of $n$ where odd parts are repeated at most twice (and there are no restrictions on the even parts).  
See \cite[\seqnum{A131945}]{OEIS} for more information.  

In \cite{MercaAM}, Merca proved the following two generating function identities:  

\begin{theorem} We have
\label{Merca_id1}
\begin{align*}
\sum_{n\geq 0} a(4n+2)q^n 
&= 2\frac{f_3f_4^3f_6^2f_{24}}{f_1^2f_2^3f_8f_{12}^2} + 4q\frac{f_6^5f_{24}^2}{f_1^3f_2^2f_{12}^3} - 2q^2\frac{f_3f_4^3f_{24}^4}{f_1^2f_2^2f_6f_8^2f_{12}^2} \\
&\ \ \ + 4q^3\frac{f_6^2f_{24}^5}{f_1^3f_2^5f_8f_{12}^3} -8q^5\frac{f_{24}^8}{f_1^3f_6f_8^2f_{12}^3}.
\end{align*}
\end{theorem}

\begin{theorem}
\label{Merca_id2}
We have 
\begin{align*}
\sum_{n\geq 0} a(4n+3)q^n 
&= 2\frac{f_4^8f_6^6f_{24}^3}{f_1f_2^7f_8^3f_{12}^7}+8q\frac{f_4^5f_6^9f_{24}^4}{f_1^2f_2^6f_3f_8^2f_{12}^8}-2q^2\frac{f_4^8f_6^3f_{24}^6}{f_1f_2^6f_8^4f_{12}^7} \\
&\ \ \ -16q^3\frac{f_4^5f_6^6f_{24}^7}{f_1^2f_2^5f_3f_8^3f_{12}^8} +8q^5\frac{f_4^5f_6^3f_{24}^{10}}{f_1^2f_2^4f_3f_8^4f_{12}^8}.
\end{align*}
\end{theorem}
In \cite{MercaAM}, Merca relied completely on the work of Radu \cite{SR} and Smoot \cite{NS} to provide ``automated'' proofs of Theorem \ref{Merca_id1} and Theorem \ref{Merca_id2}.  While these proofs are indeed correct, they do not provide any insights into the results themselves.  (In particular, the technique provides no guidance as to whether simpler generating function identities exist.)  In this case, the ``power'' of these techniques is, in essence, unnecessary.  

As corollaries of Theorem 1 and Theorem 2, Merca then noted the following:  

\begin{corollary}
\label{Merca_congs}
For all $n\geq 0$, 
\begin{align*}
a(4n+2) &\equiv 0 \pmod{2}, \textrm{\ \ and} \\
a(4n+3) &\equiv 0 \pmod{2}.
\end{align*}
\end{corollary}
At the end of \cite[Section 4]{MercaAM}, Merca indicates that a classical proof of these congruences would be very interesting. The goal of this short note is to fulfill Merca's request by providing two truly elementary (classical) proofs of Corollary \ref{Merca_congs}.  

In Section \ref{sec:ids}, we give brief elementary proofs of the following generating function results (which are much simpler than those in Theorem \ref{Merca_id1} and Theorem \ref{Merca_id2}): 

\begin{theorem}
\label{JAS_id1}
We have 
$$
\sum_{n\geq 0} a(4n+2)q^n = 2\frac{f_2f_6^2f_8^2}{f_1^4f_3f_{12}}.
$$
\end{theorem}

\begin{theorem}
\label{JAS_id2}
We have 
$$
\sum_{n\geq 0} a(4n+3)q^n = 2\frac{f_2^4f_8^2f_{12}}{f_1^5f_4^2f_6}.
$$

\end{theorem}
\noindent 
Clearly, Corollary \ref{Merca_congs} follows immediately from Theorem \ref{JAS_id1} and Theorem \ref{JAS_id2}.  

Then, in Section \ref{sec:thetafn}, we return to the congruences stated in Corollary \ref{Merca_congs} and recall that there are no squares in the arithmetic progressions $4n+2$ and $4n+3$.  With this in mind, we link the generating function for $a(n)$ with a particular theta function which then immediately gives a new proof of these divisibilities (and sheds some light as to why these particular arithmetic progressions arise).  

\section{Proofs via Simplified Generating Function Identities}
\label{sec:ids}

In order to prove Theorem \ref{JAS_id1} and Theorem \ref{JAS_id2}, we require two elementary dissection lemmas. 

\begin{lemma}
\label{lemma_f3overf1}
We have 
$$
\frac{f_3}{f_1} = \frac{f_4f_6f_{16}f_{24}^2}{f_2^2f_8f_{12}f_{48}} + q\frac{f_6f_8^2f_{48}}{f_2^2f_{16}f_{24}}.
$$
\end{lemma}

\begin{proof}
See da Silva and Sellers \cite[Lemma 1, (14)]{dSS}.  
\end{proof}

\begin{lemma}
\label{lemma_1overf1squared}
We have 
$$
\frac{1}{f_1^2} = \frac{f_8^5}{f_2^5f_{16}^2} + 2q\frac{f_4^2f_{16}^2}{f_2^5f_8}.
$$
\end{lemma}

\begin{proof}
See da Silva and Sellers \cite[Lemma 1, (11)]{dSS}.  
\end{proof} 

These are all the tools needed to complete our proofs in this section.  

\begin{proof}(of Theorem \ref{JAS_id1} and Theorem \ref{JAS_id2}) 
We begin by finding the 2--dissection of the generating function for $a(n)$, using (\ref{genfn_a}) as our starting point.  
\begin{align*}
\sum_{n\geq 0} a(n)q^n 
&= 
\frac{f_3}{f_1f_6} \\
&= 
\frac{1}{f_6}\left( \frac{f_3}{f_1} \right) \\
&= 
\frac{1}{f_6}\left( \frac{f_4f_6f_{16}f_{24}^2}{f_2^2f_8f_{12}f_{48}} + q\frac{f_6f_8^2f_{48}}{f_2^2f_{16}f_{24}} \right) \textrm{\ \ \ (Lemma \ref{lemma_f3overf1}) } \\
&= \frac{f_4f_{16}f_{24}^2}{f_2^2f_8f_{12}f_{48}} + q\frac{f_8^2f_{48}}{f_2^2f_{16}f_{24}}.
\end{align*}
Thus, we can immediately read off the following:  
\begin{align*}
\sum_{n\geq 0} a(2n)q^n &= \frac{f_2f_{8}f_{12}^2}{f_1^2f_4f_{6}f_{24}},  \\
\sum_{n\geq 0} a(2n+1)q^n &= \frac{f_4^2f_{24}}{f_1^2f_{8}f_{12}}.
\end{align*}
We now perform two additional 2--dissections.  First, 
\begin{align*}
\sum_{n\geq 0} a(2n)q^n 
&= 
\frac{f_2f_{8}f_{12}^2}{f_4f_{6}f_{24}}\left( \frac{1}{f_1^2} \right)\\
&= 
\frac{f_2f_{8}f_{12}^2}{f_4f_{6}f_{24}}\left( \frac{f_8^5}{f_2^5f_{16}^2} + 2q\frac{f_4^2f_{16}^2}{f_2^5f_8}  \right)  \textrm{\ \ \ (Lemma \ref{lemma_1overf1squared}) } \\
&= 
\frac{f_8^6f_{12}^2}{f_2^4f_4f_6f_{16}^2f_{24}} + 2q\frac{f_4f_{12}^2f_{16}^2}{f_2^4f_6f_{24}}.
\end{align*}
We then see that 
$$
\sum_{n\geq 0} a(4n+2)q^n = 2\frac{f_2f_{6}^2f_{8}^2}{f_1^4f_3f_{12}}
$$
and this proves Theorem \ref{JAS_id1}.  
Finally, 
\begin{align*}
\sum_{n\geq 0} a(2n+1)q^n 
&=
\frac{f_4^2f_{24}}{f_{8}f_{12}} \left( \frac{1}{f_1^2} \right)\\
&= 
\frac{f_4^2f_{24}}{f_{8}f_{12}}\left(  \frac{f_8^5}{f_2^5f_{16}^2} + 2q\frac{f_4^2f_{16}^2}{f_2^5f_8}  \right) \\
&= 
\frac{f_4^2f_8^4f_{24}}{f_2^5f_{12}f_{16}^2} + 2q\frac{f_4^4f_{16}^2f_{24}}{f_2^5f_8^2f_{12}}.
\end{align*}
Hence, 
$$
\sum_{n\geq 0} a(4n+3)q^n = 2\frac{f_2^4f_{8}^2f_{12}}{f_1^5f_4^2f_{6}}
$$
and this proves Theorem \ref{JAS_id2}.  
\end{proof}

\section{Proof via a Particular Theta Function}
\label{sec:thetafn} 
We begin this section by stating one additional lemma which we need as we complete our generating function manipulations below.  

\begin{lemma}
\label{div_bin_coeffs}
For all positive integers $a$ and $b$, and for all primes $p$, 
$$f_{ap}^b \equiv f_a^{bp} \pmod{p}.$$ 
\end{lemma}
\begin{proof}
This lemma follows immediately from the Binomial Theorem and the fact that, for any prime $p$ and all $1\leq i\leq p-1$,  the binomial coefficients $\binom{p}{i}$ are divisible by $p$.  
\end{proof}

We are now in a position to provide an alternative proof of Corollary \ref{Merca_congs}.  Consider
\begin{equation*}
G(q):= \frac{f_2^2f_3f_{12}}{f_1f_4f_6}
\end{equation*}
which appears in Mersmann's list of the 14 primitive eta--products which
are holomorphic modular forms of weight 1/2.  See \cite[page 30]{123MF},  \cite[Theorem 1.1]{LO13}, and \cite[\seqnum{A089801}]{OEIS} for additional details.  After some elementary calculations, we find that 
\begin{equation}
\label{Gseries}
G(q)= \sum_{k=-\infty}^\infty q^{3k^2+2k}=1+q+{q}^{5}+{q}^{8}+{q}^{16}+{q}^{21}+{q}^{33}+{q}^{40}+{q}^{56}+{q}^{65}+{q}^{85}+ \dots
\end{equation}
If we write 
$$G(q) = \sum_{n\geq 0}g(n)q^n,$$
then we see from (\ref{Gseries}) that 
\begin{equation}
\label{gformula}
g(n)=
\begin{cases}
	      1, & \text{if $n=3k^2+2k$ for some integer $k$}\\
           0, & \text{otherwise.}
\end{cases}
\end{equation}

From (\ref{genfnp}), we  see that 
$$
\frac{1}{f_{12}}G(q)  = \left(\sum_{n\geq 0}p(n)q^{12n}  \right)  \left( \sum_{n\geq 0}g(n)q^{n} \right).
$$
Moreover,  
\begin{align*}
\frac{1}{f_{12}}G(q) 
&=  
\frac{f_2^2f_3}{f_1f_4f_6} \\
&\equiv 
\frac{f_4f_3}{f_1f_4f_6} \pmod{2}   \textrm{\ \ \ (Lemma \ref{div_bin_coeffs}) }\\
&=
\frac{f_3}{f_1f_6}  \\
&= 
\sum_{n\geq 0} a(n)q^n. 
\end{align*}

Therefore, we know that  
$$
\sum_{n\geq 0} a(n)q^n \equiv \left(\sum_{n\geq 0}p(n)q^{12n}  \right)  \left( \sum_{n\geq 0}g(n)q^{n} \right) \pmod{2},
$$
so that, by convolution, we have 
\begin{equation}
\label{formula_for_a}
a(n) \equiv \sum_{k\geq 0} p(k)g(n - 12k) \pmod{2}
\end{equation}
where we define $g(m) = 0$ if $m<0$.  
In order to prove the first congruence in Corollary \ref{Merca_congs}, we want to focus on $a(4n+2)$.  Using (\ref{formula_for_a}), we see that 
$$a(4n+2) \equiv \sum_{k\geq 0} p(k)g(4n+2 - 12k) \pmod{2}$$
or 
$$a(4n+2) \equiv \sum_{k\geq 0} p(k)g(4(n-3k)+2) \pmod{2}.$$
From this congruence, we see that, if we can prove $g(4N+2) = 0$ for all $N$, then we will have proven that, for all $n\geq 0$, 
$$a(4n+2) \equiv 0 \pmod{2}.$$  So we now focus our attention on $g(4N+2)$.  

From (\ref{gformula}), we need to determine whether $4N+2 = 3k^2+2k$ for some integer $k$.  If such a solution exists, then it means that 
$$3(4N+2) + 1 = 3(3k^2+2k)+1 = (3k+1)^2$$ 
which is clearly square.  However, $3(4N+2)+1 = 12N+7 \equiv 3 \pmod{4}$, and we know that there are no squares that are congruent to 3 modulo 4.  Thus, there can be no solutions, which means $g(4N+2)=0$ for all $N$.  

Similarly, the focus of the second congruence in Corollary \ref{Merca_congs} is $a(4n+3)$, which means we need to show that $g(4N+3) = 0$ for all $N$.  Using a similar argument to the above, we can show that this is true because 
$4N+3 = 3k^2+2k$ for some integer $k$ if and only if 
$$3(4N+3) + 1 = 3(3k^2+2k)+1 = (3k+1)^2$$ 
which is clearly square.  However, $3(4N+3)+1 = 12N+10 \equiv 2 \pmod{4}$, and there are no squares congruent to 2 modulo 4.  So $g(4N+3)=0$ for all $N$, and this means that 
$$
a(4n+3) \equiv \sum_{k\geq 0} p(k)g(4(n-3k)+2) \pmod{2} \equiv 0 \pmod{2}.
$$


\end{document}